\documentclass[11pt,reqno]{amsart}
\usepackage{amssymb,amsthm,amsfonts,amstext}
\usepackage{amsmath}

\usepackage{zito}
\usepackage{calrsfs}

\usepackage[dvipsnames]{xcolor}
	
\usepackage{dsfont}

\makeatletter \@addtoreset{equation}{section}
\@addtoreset{figure}{section}
 \makeatother

\newtheorem{theorem}{Theorem}[section]
\newtheorem{lemma}[theorem]{Lemma}
\newtheorem{proposition}[theorem]{Proposition}

\theoremstyle{remark}
\newtheorem*{remark}{Remark}

\newcommand{\mc}[1]{{\mathcal #1}}

\newcommand{\bb}[1]{{\mathbb #1}}

\newcommand{\<}{\langle}
\renewcommand{\>}{\rangle}

\newcommand{\cadlag}{{c\`adl\`ag~}}
\newcommand{\eps}{\varepsilon}

\newcommand{\Ln}{\Lambda_n}
\newcommand{\On}{\Omega_n}
\newcommand{\Mp}{\mc M_+}
\newcommand{\mcK}{\mc K\!\!}

\newcommand{\slinf}{\sum_{\ell=1}^\infty}
\newcommand{\siinf}{\sum_{i=1}^\infty}
\newcommand{\sil}{\sum_{i=1}^{\ell-1}}

\title{Scaling limits of Smoluchowski particles}

\begin{document}

\author{Julian Amorim}
\address{\noindent IME-USP, Rua do Mat\~ao 1010, 05508-090, S\~ao Paulo SP, Brazil.}
\email{julian.alexandre@impa.br}

\author{Arturo Arellano}
\address{\noindent Department of Mathematics and Statistics, McGill University, Montr\' eal, Qu\' ebec, Canada.}
\email{arturo.arellanoarias@mail.mcgill.ca}

\author{Milton Jara}
\address{Instituto de Matem\'atica Pura e Aplicada, Estrada Dona Castorina 110, 22460-320  Rio de Janeiro, Brazil.}  \email{mjara@impa.br}

\begin{abstract}
We prove a law of large numbers and a functional central limit theorem for the empirical density of a Marcus-Lushnikov model. The limiting density turns out to be the solution of a Smoluchowski equation, and the fluctuations around this limit are shown to be described by an Ornstein-Uhlenbeck process with drift term given by the linearization of the Smoluchowski operator.
\end{abstract}

\maketitle

\section{Introduction}

The Smoluchowski equation is a system of differential equations that describes the evolution of the proportion of particles of a given mass, for a system of particles subjected to coagulation in pairs. This equation was introduced by Smoluchowski \cite{Smo}, who idealized the following setting. Particles of radius $r$ perform independent Brownian motions with variance proportional to $1/r$. In a low-density regime, particles of radius $r_1$ collide and merge (that is, they \emph{coagulate}) with particles of radius $r_2$ at a rate proportional to 
\[
(r_1+r_2)(1/r_1+1/r_2).
\]
Under the assumption of discrete masses for the particles, this dynamics, the effective equation for the densities $u(\ell)$ of masses $\ell=1,2,\dots$ is given by
\begin{equation}
\label{Smolint}
\tfrac{d}{dt} u_t(\ell) = \sil K(i,\ell-i) u_t(i) u_t(\ell-i) - 2 \siinf K(\ell,i) u_t(\ell) u_t(i),
\end{equation}
where $K(\ell,m) := C_0(\ell^{1/3} + m^{1/3})(\ell^{-1/3} + m^{-1/3})$. Here, the first sum corresponds to coagulation of smaller particles producing a particle of mass $\ell$ and the second sum corresponds to the loss of particles of mass $\ell$ due to coagulation with other particles to form larger particles. The class of \emph{coagulation models}, which are stochastic processes with interactions similar to the one described above, have been an active subject of study up to now. Depending on the interaction between particles, encoded by the \emph{coagulation kernel} $K(\cdot,\cdot)$, a rich phenomenology appears, which leads to applications across a wide range of areas of science. These applications include the study of aerosols in atmospheric science, liquid mixing, drop formation in rain, polymer growth and even the formation of celestial bodies in cosmology.

From a probabilistic point of view, the Smoluchowski equation is an effective equation for the evolution of the density of particles in coagulation models. Among these models, probably the simplest example is the \emph{Marcus-Lushnikov process}. In this process, two particles of masses $\ell$, $m$ merge at rate $K(\ell,m)$, forming a new particle of mass equal to $\ell+m$. In this article, the focus is the derivation of scaling limits for the density of particles in the Marcus-Lushnikov model, in particular of \emph{central limit theorems}. 

A classical work in coagulation theory is the review article \cite{Ald} about the Smoluchowski equation for different types of kernels, both in a discrete setting and in a continuous setting, and in particular for the kernels $K(\ell,m) =1$, $K(\ell,m) =\ell + m$ and $K(\ell,m) = \ell m$. This was one of the first works to emphasize the probabilistic point of view in coagulation models, more precisely through finite-volume mean-field models for the Marcus-Lushnikov process. For further details, see the surveys \cite{Ald}, \cite{Nor}, \cite{Ley}.

An important question in the study of interacting stochastic systems is to understand its limiting behavior. An example of such question is the rigorous derivation of the effective equation \eqref{Smolint} from an interacting particle system modeling the Marcus-Lushnikov process. This derivation can be recast as a weak law of large numbers for the density process, a result known in the literature as a \emph{hydrodynamic limit}. This hydrodynamic limit was derived in \cite{Nor} for a variant of the model we consider here, see also \cite{LanNgu}, \cite{FouGie}, \cite{Rez}, \cite{Arm} for other versions and generalizations of this result. Such laws of large numbers remain an active research object for coagulation models, as illustrated by more recent works such as \cite{AndIyeMag1}, with many questions open. 

Another prominent subject of research for coagulation models is the description of the phenomenon known as \emph{gelation}, which can be informally described by the following question: does there exists a time $t>0$ at which a macroscopic fraction of the total mass becomes concentrated in a single cluster? This topic has been extensively studied for the various models appearing in the literature of stochastic coagulation, as can be seen in classical works as \cite{Jeo} and \cite{EscMisPer}, as well as more recent contributions, such as \cite{AndIyeMag2}. Observe that gelation can be interpreted as loss of mass for solutions of \eqref{Smolint}, but the fine structure of gelation is a question that needs to be addressed looking at microscopic models.

Historically, Smoluchowski original model had a spatial component: particles were represented as Brownian motions in space, and the coagulation rate depended on their physical distance. In this setting, particles coagulate upon contact, leading, after suitable approximations, such as the Boltzmann-Grad limit and the Sto\ss{}zahlansatz, to an effective coagulation kernel depending only on particle sizes. We point out that the rigorous justification of such approximations is among the most important open problems in mathematical physics. For Brownian particles, the Sto\ss{}zahlansatz has been proved in the remarkable paper \cite{HamRez}.

In this work, however, we focus on a mean-field version of the model, on which spatial dependence is neglected and particles interact uniformly with one another.
We consider the case where coagulation rates are uniformy bounded and depend only on particle masses. First we recover the hydrodynamic limit results of \cite{LanNgu}, \cite{Nor}, \cite{FouGie}, \cite{Rez} using the martingale for hydrodynamic limits, as described in \cite{KipLan}. The proof of such result requires the derivation of some \emph{a priori} bounds, which are a recurrent theme in the mathematical theory of Smoluchowski equation.

Once a hydrodynamic limit for the density of particles is derived, a natural question from the probabilistic point of view is to look at \emph{fluctuations} around the hydrodynamic limit. When one looks at atypical fluctuations, the goal is to derive a \emph{lrage deviations principle}. When one looks at typical fluctuations, the goal is to derive a \emph{central limit theorem}. In this article, we describe the behavior of typical fluctuations by proving the convergence of a properly defined \emph{density fluctuation field} to the solution of an inifinite-dimensional, linear stochastic differential equation, which in particular has Gaussian solutions. This question about typical fluctuations has received far less interest from the community, and we are only aware of the results in \cite{DeaFouTan}, \cite{Kol}. In particular, our convergence result gives an answer to Open Problem 9 of \cite{Ald} in the case of bounded kernels $K$. Comparing the statement of Open Problem 9 in \cite{Ald} with our Theorem \ref{t3}, we noticed a misprint in \cite{Ald}, which nevertheless does not alter the Gaussian character of fluctuations.

The strategy of proof of both the hydrodynamic limit and the convergence of fluctuations is the same. We start from a weak formulation of the equations satisfied by the limiting objects. Before entering into the proofs of Theorems \ref{t2}, \ref{t3}, we derive some \emph{a priori} estimates needed for the aforementioned proofs. The \emph{a priori} estimates of Section \ref{s3.1} are needed both in the proof of Theorem \ref{t1} and in the derivation of the \emph{a priori} estimates of Section \ref{s4.1}. Using Dynkin formula, we show that the empirical process and the fluctuation field satisfy approximated versions of the limiting equations. The \emph{a priori} estimates allow us to show that the corresponding error terms vanish in the limit, and that the corresponding sequences of processes are tight with respect to suitable topologies.

\section{The model}

Let $\bb N := \{1,2,\dots\}$ and let $n \in \bb N$ be a scaling parameter. Let $\Ln := \{1,\dots,n\}$ and let $\bb N_0 := \{0,1,\dots\}$. We call the elements of $\Ln$ \emph{sites} and we denote them by $x,y,z,...$. Let $\On := \bb N_0^{\Ln}$ be the \emph{space of configurations}. The elements of $\On$ are denoted by $\eta, \xi, ...$. A configuration in $\Omega_n$ assigns a non-negative value $\eta_x$ to a site $x \in \Ln$, that represents the \emph{mass} of a particle at site $x$. If $\eta_x =0$, we say that site $x$ does not have a particle. 

For each $x \in \Ln$, $\delta_x$ denotes the configuration with a particle of mass $1$ at $x$ and no other particles, that is,
\[
(\delta_x)_y :=
\left\{
\begin{array}{r@{\;;\;}l}
1 & y=x,\\
0 & y \neq x.
\end{array}
\right.
\]
Let $ K: \bb N_0 \times \bb N_0 \to [0,\infty)$ be a symmetric, bounded function such that $K(0,\ell) =0$ for every $\ell \in \bb N_0$. For each $n \in \bb N$, let $L_n$ be the operator given by 
\[
L_n f(\eta) := \frac{1}{n} \sum_{x \in \Ln} \sum_{y \neq x} K(\eta_x, \eta_y) \big( f(\eta+ \eta_x(\delta_y-\delta_x))- f(\eta) \big)
\]
for every $f: \On \to \bb R$ and every $\eta \in \On$.

The operator $L_n$ defined in this way turns out to be the generator of a continuous-time Markov chain that has the following dynamics. At rate $\frac{1}{n} K(\eta_x, \eta_y)$, the particle at site $x$ jumps to site $y$, on which case the the mass of the particle at site $x$ is then added to the mass of the particle at site $y$. Since $K(0,\ell) = K(\ell,0)$ for every $\ell \in \bb N_0$, jumps occur only between sites with particles. 

\begin{remark}
The factor $\frac{1}{n}$ in the definition is chosen in such a way that the chain has a well-defined scaling limit without introducing a time rescaling.
\end{remark}

It will be useful to introduce the \emph{carr\'e du champ} operator associated to the generator $L_n$, which is the quadratic operator $\Gamma_n$ given by
\[
\Gamma_n f(\eta) := \frac{1}{n} \sum_{x \in \Ln} \sum_{y \neq x} K(\eta_x, \eta_y) \big( f(\eta+ \eta_x(\delta_y - \delta_x))- f(\eta) \big)^2
\]
for every $f: \On \to \bb R$ and every $\eta \in \On$.

For each $\ell \in \bb N_0$, let $N_\ell: \Omega_n \to \bb N_0$ be the number of particles with mass $\ell$, that is,
\[
N_\ell(\eta) := \sum_{x \in \Ln} \mathds{1}( \eta_x = \ell).
\]
Observe that
\[
\sum_{\ell =0}^\infty N_\ell =n.
\]
Therefore, the infinite vector
\[
\Big( \frac{N_\ell(\eta)}{n}; \ell \in \bb N\Big)
\]
is a subprobability measure in $\bb N_0$ for every $\eta \in \On$. 

Let $\mathds{1} \in \On$ be the \emph{monodisperse} configuration, that is,
\[
\mathds{1}_x := 1
\]
for every $x \in \Ln$, and let $(\eta(t); t \geq 0)$ be the Markov chain generated by $L_n$ with initial condition $\eta(0) = \mathds{1}$. Let $(\pi_t^n; t \geq 0)$ be the process given by
\[
\pi_t^n(\ell)  := \frac{N_\ell(\eta(t))}{n}
\]
for every $t \geq 0$ and every $\ell \in \bb N$. Observe that the dynamics conserves the total mass of the system. Since $\eta(0) = \mathds{1}$ we see that
\begin{equation}
\label{arica}
\slinf \pi_t^n(\ell) \leq 1 \text{ and } \slinf \ell \pi_t^n(\ell) =1.
\end{equation}
Therefore, $(\pi_t^n; t \geq 0)$ is a stochastic process defined on the set
\[
\Mp := \Big\{ \pi: \bb N \to [0,\infty); \slinf \pi(\ell) \leq 1 \Big\}
\]
of subprobability measures on $\bb N$.
Let us denote by $\|\cdot\|_1$ the $\ell^1$-norm in $\bb N$, that is,
\[
\|u-v\|_1 :=  \slinf |u_\ell - v_\ell|
\]
for every $u, v \in \ell^1(\bb N)$. Observe that $\Mp$ is a closed, convex subset of $\ell^1(\bb N)$. We will also consider the \emph{uniform} norm, defined as
\[
\|u\|_\infty := \sup_{\ell \in \bb N} |u_\ell|.
\]

Let $\mc D([0,\infty); \Mp)$ be the space of \cadlag paths equipped with the $J_1$-Skorohod topology. We will think about $(\pi_t^n; t \geq 0)$ as a random variable with values in $\mc D([0,\infty); \Mp)$. The law of $(\pi_t^n; t \geq 0)$ will be denoted by $\bb P_n$ and the expectation with respect to $\bb P_n$ will be denoted by $\bb E_n$.
Fix a finite time horizon $T$. Our first aim will be to establish a law of large numbers for the sequence of random variables $(\pi_t^n; t \in [0,T])_{n \in \bb N}$ with respect to $\bb P_n$.  
In order to state this result, we need to introduce the \emph{Smoluchowski equation}. 

\subsection{The Smoluchowski equation}

The Smoluchowski equation associated to the kernel $K$ is formally the system of ODEs 
\begin{equation}
\label{Smol}
\tfrac{d}{dt} u_\ell(t)  = \sil K(i,\ell-i) u_i(t) u_{\ell-i}(t) -  2\siinf K(\ell,i) u_i(t) u_\ell(t),
\end{equation}
for $\ell \in \bb N$ and $t \in [0,T]$. In order to define the Cauchy problem for \eqref{Smol} in a more rigorous way, let us define the operator $\mcK: \ell^1(\bb N) \to \ell^1(\bb N)$ as
\[
(\mcK u)_\ell := \sil K(i,\ell-i) u_i u_{\ell-i} -  2 \siinf  K(\ell,i) u_\ell  u_i
\]
for every $u \in \ell^1(\bb N)$ and every $\ell \in \bb N$. Observe that for every $u \in \ell^1(\bb N)$,
\begin{equation}
\label{copiapo}
\begin{split}
\slinf | (\mcK u)_\ell | 
		&\leq \slinf \sil K(i,\ell-i) |u_i||u_{\ell-i}| + 2 \slinf \sil K(\ell,i) |u_i||u_\ell| \\
		&\leq \|K\|_\infty \Big( \slinf \sil |u_i||u_{\ell-i}| + 2  \slinf \sil |u_i||u_\ell| \Big) \\
		&\leq  \|K\|_\infty \Big( \slinf \siinf |u_i||u_\ell| +2 \| u\|_1^2 \Big) \leq 3 \|K\|_\infty \|u\|_1^2,
\end{split}
\end{equation}
and in particular $\mcK$ is well defined. In a similar way, 
\begin{equation}
\label{caldera}
\| \mcK u - \mcK v \|_1 \leq 3 \|K\|_\infty ( \|u\|_1 + \|v\|_1) \|u - v \|_1
\end{equation}
for every $u, v \in \ell^1(\bb N)$. 

Let $T >0$ be fixed. We say that a trajectory $u \in \mc C([0,T]; \Mp)$ is a \emph{weak solution} of \eqref{Smol} if for every $t \geq 0$,
\[
u(t) = u(0) + \int_0^t \!\! \mcK u(s) ds.
\]
We say that $u(0)$ is the \emph{initial condition} of this solution. In the case on which $u(0, \ell) = \delta(\ell=1)$, we say that the initial condition is monodisperse. We have the following uniqueness result:

\begin{proposition}
\label{t1}
For every $T >0$ and every $ \overline{u} \in \Mp$, there exists at most one solution of \eqref{Smol} in $\mc C([0,T]; \Mp)$ with initial condition $\bar{u}$.
\end{proposition}

\begin{proof}
The proof of this result is very simple. From \eqref{caldera}, $\mcK$ is uniformly Lipschitz in every bounded set of $\ell^1(\bb N)$. Let $a >1$ and let
\[
B:= \{ u \in \ell^1(\bb N); \|u\|_1 \leq a \}.
\]
Picard-Lindel\"of theorem guarantees local existence and uniqueness of solutions of \eqref{Smol} for initial conditions in the interior of $B$. Observe that $\Mp$ belongs to the interior of $B$. By hypothesis, the solutions of \eqref{Smol} do not leave $\Mp$, and therefore local uniqueness also implies global uniqueness. 
\end{proof}

\begin{remark}
Observe that we have stated only the uniqueness of solutions of \eqref{Smol}, and not the existence. The existence will be a consequence of our convergence theorem, stated in Theorem \ref{t2}. In the literature, it is possible to find much better existence and uniqueness results. We stated Proposition \ref{t1}  in this way for two reasons. First, it is easy to prove, and second, it is exactly what we need in order to prove the convergence stated in Theorem \ref{t2} below.
\end{remark}

\subsection{Hydrodynamic limit and Gaussian fluctuations} Now we can state what we understand as the \emph{hydrodynamic limit} of the sequence of processes $(\pi_t^n; t \geq 0)_{n \in \bb N}$.

\begin{theorem} 
\label{t2}
For every $T >0$, 
\[
\lim_{n \to \infty} \pi_t^n = u(t)
\]
in probability with respect to $\bb P_n$ and with respect to the $J_1$-Skorohod topology of $\mc D([0,T]; \Mp)$, where $u \in \mc C([0,T]; \Mp)$ is the unique solution of \eqref{Smol} with initial condition $u(0) = \delta_1$.
\end{theorem}

\begin{remark}
This theorem also holds for more general initial conditions; we have chosen the monodisperse initial condition $\delta_1$ for simplicity.
\end{remark}

Observe that Theorem \ref{t2} can be understood as a law of large numbers for the random trajectories $(\pi_t^n; t \in [0,T])_{n \in \bb N}$. When the limiting function satisfies a finite-dimensional ODE, Theorem \ref{t2} is known in the literature as a \emph{fluid limit}. When $u$ satisfies a PDE, Theorem \ref{t2} is known in the literature as a \emph{hydrodynamic limit}. Since Smoluchowski equation requires functional techniques in order to be solved, we consider it to be closer to a PDE than an ODE, and for this reason we call Theorem \ref{t2} the \emph{hydrodynamic limit} for the Smoluchowski equation. Results similar to Theorem \ref{t2} are sometimes called \emph{mean-field limits}.

Whenever a law of large number is established, a natural question is to look whether the \emph{fluctuations} around the hydrodynamic limit satisfy a convergence theorem. For atypical fluctuations, such result is known as \emph{large deviations principle}, and for typical fluctuations, such results are called \emph{central limit theorems}. Our second result describes the asymptotic behavior of typical fluctuations.  Let us define the fluctuation process $(\xi_t^n; t \geq 0)$ as
\[
\xi_t^n(\ell) := \sqrt{n} ( \pi_t^n(\ell) - u_t(\ell))
\]
for every $t \geq 0$ and every $\ell \in \bb N$. Observe that the trajectories of $(\xi_t^n; t \geq 0)$ belong to the space $\mc D([0,\infty); \ell^1(\bb N))$. In order to state the corresponding convergence theorem, we need to define the limiting diffusion.

\subsection{The limiting SDE}

Let $\mc A: \ell^1(\bb N) \times \ell^1(\bb N) \to \ell^1(\bb N)$ be given by
\[
\mc A(u,v)(\ell) := \sil K(i,\ell-i) \big( u_i v_{\ell-i} + u_{\ell-i} v_i \big) 
		- 2 \slinf K(\ell,i) \big( u_\ell v_i + v_\ell u_i \big)
\]
for every $u,v \in \ell^1(\bb N)$ and every $\ell \in \bb N$. Observe that
\begin{equation}
\label{huasco}
\big\| \mc A(u,v) \big\|_1 \leq 6 \|K\|_\infty \|u\|_1 \|v\|_1,
\end{equation}
and in particular $\mc A$ is well defined. Let $Q: \ell^1(\bb N) \times \ell^\infty(\bb N) \to \bb R$ be given by
\begin{equation}
\label{Q}
Q(u;f) := \sum_{i,j \in \bb N} K(i,j) u_i u_j \big( f_{i+j}-f_i - f_j)^2
\end{equation}
for every $u \in \ell^1(\bb N)$ and every $f \in \ell^\infty(\bb N)$. Observe that
\begin{equation}
\label{Q2}
\big| Q(u;f) \big| \leq 9 \|K\|_\infty \|u\|_1^2 \|f\|_\infty^2
\end{equation}
and in particular $Q$ is well defined. 
For $u \in \ell^1(\bb N)$ and $f \in \ell^\infty(\bb N)$, we will write
\[
\<u,f\> := \slinf u(\ell) f(\ell).
\]
We say that a process $(\xi_t; t \geq 0)$ with trajectories in $\mc C([0,T]; \bb R)$ is a \emph{weak solution} of 
\begin{equation}
\label{SDE}
d \xi_t = \mc A(u_t, \xi_t) dt + Q(u_t)^{1/2} d B_t
\end{equation}
if there is a constant $C = C(T)$ such that $\bb E_n [ \|\xi_t\|_1] \leq C$ for every $t \in [0,T]$ 
and if for every $f \in \mc C^1([0,T]; \ell^\infty(\bb N))$, the process $(\mc M_t(f); t \in [0,T])$ given by
\begin{equation}
\label{weak}
\mc M_t(f) := \<\xi_t, f_t\> - \< \xi_0, f_0 \> - \int_0^t \big(\< \xi_s, \partial_s f_s\> + \< \mc A(u_s,\xi_s), f_s \> \big)ds
\end{equation}
for every $t \in [0,T]$, is a continuous martingale of quadratic variation
\[
\<\mc M(f)\>_t = \int_0^t Q(u_s, f_s) ds.
\]
We have the following result:

\begin{proposition}
\label{p2}
Two solutions of \eqref{SDE} with the same initial condition have the same law.
\end{proposition}

\begin{proof}[Sketch of proof] 
Observe that
\[
\big\< \mc A(u_s, \xi_s), f_s\big> = \big\< \xi_s, \mc A^\ast ( u_s,f_s) \big\>,
\]
where
\[
\mc A^\ast (u,f)(\ell) = 2 \siinf K(\ell,i) \big( f_{\ell+i} -f_\ell -f_i \big) u_i.
\]
Therefore,
\[
\big\| \mc A^\ast (u,f) \big\|_\infty \leq 6 \|K\|_\infty  \|u\|_1 \|f\|_\infty
\]
and in particular $f \mapsto \mc A^\ast(u_t,f)$ is uniformly continuous in $[0,T]$. This means that the \emph{Fokker-Planck equation}
\[
\partial_s f_s+ \mc A^\ast (u_s,f_s) = 0 
\]
with final condition $f_t =g$ has a solution in $\mc C^1([0,t]; \ell^\infty(\bb N))$. Using this solution as a test function in \eqref{weak}, we see that
\[
\bb E [\<\xi_t,g\>] = \<\xi_0,f_0\>,
\]
which characterizes the law of $\xi_t$. For $0 \leq s < t$, we also see that
\[
\bb E [\<\xi_t,g\>|\mc F_s] = \<\xi_s,f_s\>,
\]
from where the law of $(\xi_t, \xi_s)$ is characterized as well. All finite-dimensional laws of $(\xi_t; t \in [0,T])$ can be characterized recursively, which shows uniqueness in law.
\end{proof}

We will prove the following result. 

\begin{theorem}
\label{t3}
For every $T >0$, the sequence $(\xi_t^n; t \in [0,T])$ converges in law with respect to the $J_1$-Skorohod topology of $\mc D([0,T]; \ell^1(\bb N))$ to the solution of the system of SDEs
\begin{equation}
\begin{split}
d \xi_t
		&= \mc A(\xi_t,u_t) dt + Q(u_t)^{1/2} d B_t
\end{split}
\end{equation}
with initial condition $\xi_0=0$.
\end{theorem}

\section{Proof of Theorem \ref{t2}}

The proof of Theorem \ref{t2} follows the classical script of proofs of functional limits of stochastic processes. First we proof tightness of the sequence $(\pi_t^n; t \in [0,T])_{n \in \bb N}$. Then we proof the every limit point is concentrated on solutions of \eqref{Smol}. After that, the proof follows at once from Proposition \ref{t1}. Before starting the proof, we need to prove some \emph{a priori} bounds for the moments of $\pi_t^n$.

\subsection{A priori estimates}
\label{s3.1}

For $p \geq 1$ and $t \geq 0$, let 
\[
M_p(t) := \bb E_n \Big[ \slinf \ell^p \pi_t^n(\ell)\Big]
\]
be the $p$-th moment of the measure $\pi_t^n$ and observe that
\[
M_p(t) = \bb E_n \Big[ \frac{1}{n} \sum_{x \in \Lambda_n} \eta_x(t)^p \Big].
\]
We start estimating $M_2(t)$, for which we use the differential Dynkin formula. Observe that
\[
\begin{split}
L_n \Big( \frac{1}{n} \sum_{x \in \Ln} \eta_x^2 \Big)
		&= \frac{1}{n^2} \sum_{x \in \Ln} \sum_{y \neq x} K(\eta_x, \eta_y) \big( (\eta_x+\eta_y)^2 -\eta_x^2 - \eta_y^2 \big)\\
		&= \frac{2}{n^2} \sum_{x \in \Ln} \sum_{y \neq x} K(\eta_x, \eta_y)  \eta_x \eta_y \leq 2 \|K\|_\infty \Big( \frac{1}{n} \sum_{x \in \Ln} \eta_x \Big)^2\\
		&\leq 2 \|K\|_\infty.
\end{split}
\]
Therefore, $M_2'(t) \leq 2 \|K\|_\infty$, from where 
\begin{equation}
\label{quillota}
M_2(t) \leq 1 + 2 \|K\|_\infty t
\end{equation}
for every $t \geq 0$. For $p \geq 2$, we have a similar estimate:

\begin{lemma}
\label{l1} 
For every $p \geq 1$ there exists a finite constant $C(p)$ such that
\[
M_p(t) \leq \big(1 + C(p) \|K\|_\infty t \big)^{p-1}
\]
for every $t \geq 0$.
\end{lemma}

\begin{proof}
We start proving the estimate for $p \in \bb N$ by induction. For $p=2$, $M_p(t) \leq 1 +2 \|K\|_\infty t$, and the lemma holds with $C(p) =2$. Assume the lemma holds for $p-1$ with constant $C(p-1) = \frac{2^{p-2}-2}{p-2}$. We have that
\[
\begin{split}
L_n \Big( \frac{1}{n} \sum_{x \in \Ln} \eta_x^p \Big) 
		&= \frac{1}{n^2} \sum_{x \in \Ln} \sum_{y \neq x} K(\eta_x, \eta_y) \big( (\eta_x + \eta_y)^p - \eta_x^p -\eta_y^p \big)\\
		&\leq \frac{\|K\|_\infty}{n^2} \sum_{x,y \in \Ln} \sum_{i=1}^{p-1} \binom{p}{i} \eta_x^{p-i} \eta_y^i \\
		&= \|K\|_\infty \Big( \frac{2p}{n} \sum_{x \in \Ln} \eta_x^{p-1} + \frac{1}{n^2} \sum_{x,y \in \Ln} \sum_{i=2}^{p-2} \binom{p}{i} \eta_x^{p-i} \eta_y^i \Big)
\end{split}
\]
By Young's inequality with $q = \frac{p-2}{p-i-1}$ and $q^\ast = \frac{p-2}{p-i}$, we see that
\[
\eta_x^{p-i} \eta_y^i \leq \Big(\frac{p-i-1}{p-2}\Big) \eta_x^{p-1} \eta_y + \Big( \frac{i-1}{p-2} \Big) \eta_x \eta_y^{p-1},
\]
from where
\[
\begin{split}
L_n \Big( \frac{1}{n} \sum_{x \in \Ln} \eta_x^p \Big)  
		&\leq \|K\|_\infty \Big( 2p + \sum_{i=2}^{p-2}  \binom{p}{i} \Big) \Big( \frac{1}{n} \sum_{x \in \Ln} \eta_x^{p-1} \Big)\\
		&\leq (2^p-2) \|K\|_\infty \Big( \frac{1}{n} \sum_{x \in \Ln} \eta_x^{p-1} \Big)
\end{split}
\]
and therefore $M_p'(t) \leq (2^p-2) \|K\|_\infty M_{p-1}(t)$. Observe that
\[
1+ a((1+t)^{p-1}-1) \leq (1+at)^{p-1}
\]
for every $a \geq 1$, every $p \geq 1$ and every $t \geq 0$. Using the inductive hypothesis and this estimate, we see that 
\[
\begin{split}
M_p(t) 
		& \leq 1 + (2^p-2) \|K\|_\infty \int_0^t M_{p-1}(s) ds\\
		& \leq 1 + (2^p-2) \|K\|_\infty \int_0^t \big(1 + C(p-1) \|K\|_\infty s \big)^{p-2} ds\\
		& \leq 1 + \frac{2^p-2}{C(p-1) (p-1)}  \big( (1 + C(p-1) \|K\|_\infty t)^p -1\big)\\
		&\leq (1+ C(p) \|K\|_\infty t)^p
\end{split}
\]
for $C(p) = \frac{2^p-2}{p-1}$, if $C(p) \geq C(p-1)$. We can check that $p \mapsto \frac{2^p-2}{p}$ is increasing for $p \geq 1+(\log 2)^{-1} = 2,\!442...$, which shows that $C(p) \geq C(p-1)$ for $p \geq 4$. Observe that $C(2) = 2$, $C(3) = 3$ and $C(4) = 4,\!666...$, from where we conclude that the inductive step holds for every $p \geq 3$, which proves the lemma for $p \in \bb N$ and $C(p) = \frac{2^p-2}{p-1}$. The lemma follows for $ p \geq 1$ arbitrary by interpolation, with a possibly different constant $C(p)$.
\end{proof}

\subsection{Tightness}
\label{s2.2}
In order to prove tightness of $(\pi_t^n; t \in [0,T])_{n \in \bb N}$, we will use \emph{Aldous criterion}, see \cite[Theorem 16.10]{Bil} and \cite[Proposition C.2]{JarMen}. We need to prove that
\begin{itemize}
\item[(A1)] for every $t \in [0,T]$ and every $\eps >0$, there exists a compact set $B(\eps,t) \subseteq \Mp$ such that 
\[
\limsup_{n \to \infty} \bb P_n( \pi_t^n \notin B(\eps,t) ) \leq \eps,
\]
\item[(A2)] for every $\eps >0$,
\[
\lim_{\delta \downarrow 0} \limsup_{n \to \infty} \sup_{\gamma \leq \delta} \sup_{\tau \in \mc T_T} \bb P_n \big( \|\pi^n_{\tau+\gamma}-\pi^n_\tau \|_1 > \eps \big) =0,
\]
\end{itemize}
where $\mc T_T$ is the set of stopping times bounded by $T$. Recall that we are using the $\ell^1$-topology on $\Mp$. With respect to this topology, $\Mp$ is not precompact, but for each $M >0$, the sets
\[
B_M := \Big\{ u \in \Mp; \slinf \ell u(\ell) \leq M \Big\}
\]
are compact in $\ell^1(\bb N)$. Therefore,
\[
\bb P_n ( \pi_t^n \notin B_M) \leq \frac{1}{M} \bb E_n \Big[ \frac{1}{n} \slinf \ell \pi_t^n(\ell) \Big] \leq \frac{1}{M},
\]
which shows (A1).

In order to prove (A2), we will use Dynkin formula for the functions $\pi_t^n(\ell)$. Since we will make repeated use of this formula, we will state it with some care. For every $\ell \in \bb N$, the process $(\mc M_t^n(\ell); t \geq 0)$ given by
\begin{equation}
\label{laserena}
\mc M_t^n(\ell) := \sqrt n \Big(\pi_t^n(\ell) - \pi_0^n(\ell) - \int_0^t \!\! L_n \pi_s^n(\ell) ds \Big)
\end{equation}
is a martingale. We will see in the next section in a precise way that the factor $\sqrt n$ makes $\mc M_t^n$ of order $\mc O(1)$. The quadratic variation of $(\mc M_t^n; t \geq 0)$ is given by
\begin{equation}
\label{coquimbo}
\< \mc M^n(\ell)\>_t = n\int_0^t \!\! \Gamma_n \pi_s^n(\ell) ds.
\end{equation}
Recall that $\pi_t^n(\ell) = \frac{N_\ell(\eta(t))}{n}$. We have that
\[
\begin{split}
L_n \Big( \frac{N_\ell}{n} \Big) 
		&= \frac{1}{n^2} \Big( \sil K(i,\ell-i) N_i \big( N_{\ell-i} - \delta(2i = \ell) \big)
		 - 2 \siinf K(i,\ell)  N_\ell\big( N_i - \delta(i=\ell) \big) \Big)\\
		 &\leq \|K\|_\infty\Big( \sil \frac{N_i}{n} \frac{N_{\ell-i}}{n} + 2 \siinf \frac{N_\ell}{n} \frac{N_i}{n} \Big),
\end{split}
\]
from where
\begin{equation}
\label{aconcagua}
\begin{split}
\| L_n \pi_t^n \|_1 
		&\leq \|K\|_\infty \Big( \slinf \sil \pi_t^n(i) \pi_t^n(\ell-i) + 2\slinf \siinf \pi_t^n(\ell) \pi_t^n(i) \Big)\\
		&\leq 3 \|K\|_\infty \|\pi_t^n\|_1^2 \leq 3 \|K\|_\infty.
\end{split}
\end{equation}
Observe as well that 
\begin{equation}
\label{quilpue}
\begin{split}
L_n \frac{N_\ell}{n} 
		&= \frac{1}{n^2} \Big( \sil K(i,\ell-i) N_i \big( N_{\ell-i} - \delta(2i = \ell) \big)
		 - 2 \siinf K(i,\ell)  N_\ell\big( N_i - \delta(i=\ell) \big) \Big)\\
		& = \mcK \pi^n(\ell) + \frac{1}{n} \big(2 K(\ell,\ell) \pi^n(\ell) - K(\ell/2,\ell/2) \pi^n(\ell/2) \big).
\end{split}
\end{equation}
Here and below we use the conventions $K(\ell/2,\ell/2)=0$ and $\pi^n(\ell/2) =0$ if $\ell/2 \notin \bb N$. In order to simplify the notation, let $\mc R: \Mp \to \ell^1(\bb N)$ be the operator given by
\[
\mc R u(\ell) := 2 K(\ell,\ell) u(\ell) - K(\ell/2, \ell/2) u(\ell/2)
\]
for every $u \in \Mp$ and every $\ell \in \bb N$. With this notation, we see that
\begin{equation}
\label{portillo}
L_n \pi_t^n(\ell) = \mcK \pi_t^n(\ell) + \frac{1}{n} \mc R \pi_t^n(\ell).
\end{equation}
Observe that for every $u \in \Mp$,
\begin{equation}
\label{sanfelipe}
\mc R u(\ell) \leq \|K\|_\infty \big( 2 u(\ell) + u(\ell/2) \big),
\end{equation}
from where
\begin{equation}
\label{colina}
\| \mc R u\|_1 \leq 3 \|K\|_\infty \|u\|_1
\end{equation}
Therefore,
\begin{equation}
\label{petorca}
\big| L_n \pi_t^n(\ell) - \mcK \pi_t^n(\ell) \big| \leq \frac{3 \|K\|_\infty}{n},
\end{equation}
and also
\begin{equation}
\label{laligua}
\big\|  L_n \pi_t^n - \mcK \pi_t^n \big\|_1 \leq \frac{3 \|K\|_\infty}{n}.
\end{equation}
These bounds will not be needed to show (A2), but they will be needed in order to show the convergence result of Theorem \ref{t2}.
Observe as well that
\begin{equation}
\label{zapallar}
\begin{split}
\Gamma_n \pi^n(\ell) 
		&= \frac{1}{n^3}\Big( \sil K(i,\ell-i) N_i \big( N_{\ell-i} - \delta(2i = \ell) \big) + 2 \siinf K(i,\ell)  N_\ell\big( N_i - \delta(i=\ell) \big) \Big)\\
		&\leq \frac{\|K\|_\infty}{n} \Big( \sil \pi^n(i) \pi^n(\ell-i) + 2 \pi^n(\ell)\Big).
\end{split}		
\end{equation}
Formula \eqref{laserena} can be rewritten as the vectorial identity 
\begin{equation}
\label{olmue}
\pi_{\tau+\gamma}^n - \pi_{\tau}^n = \int_{\tau}^{\tau+\gamma} \!\! L_n \pi_s^n ds + \frac{\mc M_t^n}{\sqrt n},
\end{equation}
from where we see that
\[
\begin{split}
\|\pi_{\tau+\gamma}^n - \pi_{\tau}^n \|_1 
		&\leq \int_{\tau}^{\tau+\gamma} \!\! \| L_n \pi_s^n \|_1 ds + \frac{1}{n} \| \mc M_{\tau+\gamma}^n - \mc M_{\tau}^n \|_1\\
		&\leq 3 \gamma \|K\|_\infty + \frac{1}{n} \| \mc M_{\tau+\gamma}^n - \mc M_{\tau}^n \|_1,\\
\end{split}
\]
where we used \eqref{aconcagua}.
In other words, the integral part of \eqref{olmue} satisfies (A2) almost surely. We are left to estimate $\| \mc M_{\tau+\gamma}^n - \mc M_{\tau}^n \|_1$. We will use an estimate which is not optimal; but it is simple. A more refined argument will be needed in order to derive the fluctuations result. 

Observe that $\mc M_t^n(\ell) =0$ if $\ell > n$ and also observe that
\begin{equation}
\label{llolleo}
\Gamma_n \pi_t^n(\ell) \leq \frac{3 \|K\|_\infty}{n}.
\end{equation}
Therefore,
\[
\begin{split}
\bb E_n \big[ \| \mc M_{\tau+\gamma}^n - \mc M_{\tau}^n \|_1\big] 
		&= \sum_{\ell=1}^n \bb E_n \big[ |\mc M_{\tau+\gamma}^n(\ell) - \mc M_{\tau}^n(\ell) | \big] \leq \sum_{\ell=1}^n \bb E_n \big[ \<\mc M^n(\ell)\>_{\tau}^{\tau+\gamma} \big]^{1/2} \\
		&\leq \sqrt n \Big( \sum_{\ell=1}^n \bb E_n \big[ \<\mc M^n(\ell)\>_{\tau}^{\tau+\gamma} \big] \Big)^{1/2} \leq \sqrt{3 n \gamma \|K\|_\infty}.
\end{split}
\]
Putting these estimates together, we see that
\[
\bb P_n \big(  \| \pi_{\tau+\gamma}^n - \pi_{\tau}^n \|_1 > \eps \big)
		\leq \frac{\bb E_n \big[  \| \pi_{\tau+\gamma}^n - \pi_{\tau}^n \|_1 \big]}{ \eps} \leq \frac{ 3 \gamma \|K\|_\infty + \sqrt{\frac{3 \gamma}{n} \|K\|_\infty }}{\eps},
\]
from where (A2) follows.

\subsection{Characterization of limit points} 

Now that we know that $(\pi_t^n; t \in [0,T])_{n \in \bb N}$ is tight, we proceed to characterize its limit points. Recall that \eqref{laserena} states that for every $\ell \in \bb N$,
\[
\pi_t^n(\ell) = \pi_0^n(\ell) + \int_0^t \!\!L_n \pi_s^n(\ell) ds + \frac{\mc M_t^n(\ell)}{\sqrt n}.
\]
Using \eqref{llolleo}, we have that
\[
\bb E_n \Big[ \sup_{0 \leq t \leq T} |\mc M_t^n(\ell)|^2 \Big] \leq 4 \bb E_n\big[ \mc M_T^n(\ell)^2 \big] \leq 4 \bb E_n \big[ \< \mc M^n(\ell) \>_T \big] \leq 12 \|K\|_\infty T.
\]
Therefore, $(\frac{1}{\sqrt n}\mc M_t^n(\ell); t \in [0,T])_{n \in \bb N}$ converges to 0 in $L^2(\bb P_n)$ with respect to the uniform topology.

Recall now \eqref{petorca}. We see that
\[
\Big| \int_0^t \!\!\big( L_n \pi_s^n(\ell) - \mcK \pi_s^n(\ell) \big) ds \Big| \leq \frac{3\|K\|_\infty t}{n}.
\]
Therefore, 
\[
\lim_{n \to \infty} \Big( \pi_t^n(\ell) - \pi_0^n(\ell) - \int_0^t \!\! \mcK \pi_s^n(\ell)ds \Big) = 0
\]
in $L^2( \bb P_n)$. Observe as well that the jumps of the process $(\pi_t^n; t \geq 0)$ are of size at most $\frac{3}{n}$. Now we are ready to characterize the limit points of $(\pi_t^n; t \in [0,T])_{n \in \bb N}$. Let $(u_t; t \in [0,T])$ be a limit point. First we observe that the size of the largest jump is a continuous function with respect to the $J_1$-topology. Therefore, $(u_t; t \in [0,T])$ has continuous trajectories (the largest jump has size zero), i.e., its law is concentrated in $\mc C([0,T]; \Mp)$. To follow, we observe that the projection operator $u \mapsto u_t(\ell)$ is not continuous in $\mc D([0,T]; \Mp)$, but it is continuous in $\mc C([0,T]; \Mp)$. Therefore, 
\[
\lim_{n \to \infty} \pi_t^n(\ell) = u_t(\ell)
\]
in law with respect to $\bb P_n$ for every $t \in [0,T]$ and every $\ell \in \bb N$. To finish, we observe that integration is a continuous operation in $\mc D([0,T]; \Mp)$. We conclude that
\[
u_t(\ell) = u_0(\ell) + \int_0^t \mcK u_s(\ell) ds 
\]
for every $t \in [0,T]$ and every $\ell \in \bb N$. In principle, this identity only holds almost surely, and therefore it does not hold simultaneously for every $t \in [0,T]$. However, by the continuity of $u$, we can assume that with probability 1, this identity holds for every $t \in [0,T]$. We conclude that $u$ is a solution of \eqref{Smol}, and the uniqueness result of Proposition \ref{t1} ends the proof of Theorem \ref{t2}.

\section{The fluctuations}

In this section we will prove Theorem \ref{t3}. As we did in the proof of Theorem \ref{t2}, we start deriving some \emph{a priori} estimates for the process $(\xi_t^n; t \in [0,T])$. 

\subsection{A priori estimate in $L^2(\bb P_n)$}
\label{s4.1}

The aim of this section is to prove the following moment bound for the $\ell^1$-norm of $\xi_t^n$:
\begin{lemma}
\label{l2} For every $T >0$ there exists a finite constant $C= C(T, \|K\|_\infty)$ such that
\[
\bb E_n \big[ \| \xi_t^n\|_1^2 \big] \leq C
\]
for every $n \in \bb N$ and every $t \in [0,T]$.
\end{lemma}

\begin{proof}
Recall that $\xi_t^n = \sqrt{n}(\pi_t^n - u_t)$ and observe that $\xi_0^n=0$. Therefore, \eqref{olmue} can be rewritten as
\begin{equation}
\label{nos}
\begin{split}
\xi_t^n
		&= \int_0^t \sqrt{n} \big( L_n \pi_s^n - \tfrac{d}{ds} u_s \big)ds +  \mc M_t^n\\
		&= \int_0^t \sqrt{n} \big( L_n \pi_s^n - \mcK u_s \big)ds +  \mc M_t^n.
\end{split}
\end{equation}
Using \eqref{portillo} and \eqref{laligua}, we see that
\[
\|\xi_t^n\|_1 \leq \int_0^t \sqrt n \| \mcK \pi_s^n - \mcK u_s\|_1 ds + \|\mc M_t^n\|_1 + \frac{3 \|K\|_\infty t}{\sqrt n}.
\]
Using \eqref{caldera}, we see that
\begin{equation}
\label{buin}
\sqrt n \|\mcK \pi_s^n - \mcK u_s\|_1 \leq 3 \|K\|_\infty \sqrt{n} \big( \|\pi_s^n\|_1 + \|u_s\|_1 \big) \|\pi_s^n - u_s\|_1 \leq 6 \|K\|_\infty \|\xi_s^n\|_1,
\end{equation}
from where
\[
\|\xi_t^n\|_1 \leq \int_0^t 6 \|K\|_\infty \|\xi_s^n\|_1 ds +  \|\mc M_t^n\|_1 + \frac{3 \|K\|_\infty t}{\sqrt n}.
\]
Using Gronwall inequality, we see that
\[
\| \xi_t^n\|_1 \leq \int_0^t e^{6\|K\|(t-s)} \Big( \|\mc M_s^n\|_1 + \frac{3 \|K\|_\infty s}{\sqrt n}\Big)ds.
\]
Therefore, in order to prove the lemma, we are left to derive an estimate for
\[
\bb E_n \big[ \|\mc M_t^n\|_1^2]
\]
that is uniform in $n \in \bb N$ and $t \in [0,T]$. From \eqref{llolleo}, we see that
\[
\bb E_n \big[ \mc M_t^n(\ell)^2 \big] \leq 3\|K\|_1 t,
\]
which is not good enough for our purposes, since we need to sum over $\ell$. In order to obtain a better  estimate, we proceed as follows. Using Minkowski and Cauchy-Schwarz inequalities, we see that
\[
\begin{split}
\bb E_n\big[  \|\mc M_t^n\|_1^2 \big] 
		&\leq \Big( \slinf \bb E_n \big[ \mc M_t^n(\ell)^2\big]^{1/2} \Big)^2
		\leq  \Big( \slinf \frac{1}{\ell^2} \Big) \Big( \slinf \ell^2 \bb E_n \big[ \mc M_t^n(\ell)^2 \big] \Big)\\
		&\leq C \slinf \ell^2 \bb E_n \Big[ \int_0^t \Gamma_n \pi_s^n(\ell) ds \Big]\\
		&\leq C \| K \|_\infty \bb E_n \Big[ \int_0^t \slinf \ell^2 \Big(  \sil \pi_s^n(i) \pi_s^n(\ell-i) + 2 \pi_s^n(\ell) \Big) ds \Big]\\
		&\leq C \|K\|_\infty \int_0^t \bb E_n \Big[ \slinf \siinf (\ell+i)^2 \pi_s^n(i) \pi_s^n(\ell) + 2 \slinf \ell^2 \pi_s^n(\ell) \Big] ds\\
		&\leq  C \|K\|_\infty \int_0^t \big( 4 M_2(s) + 2 \big) ds. 
\end{split}
\]
Here we have used \eqref{arica} in various steps.
Using the \emph{a priori} estimate of Lemma \ref{l1}, the lemma is proved.
\end{proof}

\subsection{A priori estimate in $L^1(\bb P_n)$} The aim of this section is to prove the following estimate:

\begin{lemma}
\label{l3}
For every $T >0$ there exists a constant $C= C(T, \|K\|_\infty)$ such that
\[
\bb E_n \Big[ \slinf \ell |\xi_t^n(\ell)| \Big] \leq C
\]
for every $n \in \bb N$ and every $t \in [0,T]$.
\end{lemma}

\begin{proof}
Let us introduce the notation
\[
\|v\|_{1,1} := \slinf \ell |v(\ell)|
\]
for $v \in \ell^1(\bb N)$. From \eqref{nos},
\[
\|\xi_t^n\|_{1,1} \leq  \int_0^t \sqrt n \big\| L_n \pi_s^n - \mcK u_s \big\|_{1,1} ds + \|\mc M_t^n\|_{1,1}.
\]
From \eqref{arica}, $\|\pi_t^n\|_{1,1} =1$. Observe that
\[
\begin{split}
\sqrt n \big\|L_n \pi_s^m - \mcK \pi_s^n \big\|_{1,1}
		&= \frac{1}{\sqrt n} \big\|\mc R \pi_s^n\big\|_{1,1} \leq \frac{1}{\sqrt n} \slinf \ell \big| \mc R\pi_s^n(\ell) \big|\\
		&\leq \frac{\|K\|_\infty}{\sqrt n} \slinf \ell \big( 2 |\pi_t^n(\ell)| + |\pi_t^n(\ell/2)|\big) \leq \frac{4 \|K\|_\infty\|\pi_t^n\|_{1,1}}{\sqrt n}\\
		&\leq \frac{4 \|K\|_\infty}{\sqrt n}.
\end{split}
\]
We also have that
\[
\begin{split}
\sqrt n \| \mcK \pi_t^n - \mcK u_t \|_{1,1} 
		&\leq \|K\|_\infty \Big( \slinf \sil \ell \big( \pi_t^n(i) |\xi_t^n(\ell-i)| + |\xi_t^n(i)| u_t(\ell-i) \big)\\
		&\quad \quad \quad \quad \quad + 2 \slinf \siinf \ell\big( \pi_t^n(\ell) |\xi_t^n(i)| + |\xi_t^n(\ell)| u_t(i) \big) \Big)\\
		&\leq \|K\|_\infty \Big( \slinf \siinf (\ell+i) \big( \pi_t^n(i) |\xi_t^n(\ell-i)| + |\xi_t^n(i)| u_t(\ell-i) \big)\\
		&\quad \quad \quad \quad \quad +2\big( \|\pi_t^n\|_{1,1} \|\xi_t^n\|_1 + \|\xi_t^n\|_{1,1} \|u_t\|_1 \big) \Big)\\
		&\leq 8 \|K\|_\infty \|\xi_t^n\|_{1,1}.
\end{split}
\]
Therefore,
\[
\|\xi_t^n\|_{1,1} \leq \int_0^t 8 \|K\|_\infty \|\xi_s^n\|_{1,1} ds + \frac{4\|K\|_\infty t}{\sqrt n} + \|\mc M_t^n\|_{1,1}.
\]
Using Gronwald inequality, we obtain the estimate
\begin{equation}
\label{curico}
\|\xi_t^n\|_{1,1} \leq \int_0^t e^{8\|K\|_\infty (t-s)} \Big( \|\mc M_s^n \|_{1,1} + \frac{4\|K\|_\infty s}{\sqrt n} \Big) ds.
\end{equation}
Consequently, in order to establish the estimate of the lemma, is it sufficient to estimate   $\bb E_n [ \|\mc M_t^n\|_{1,1}]$. By \eqref{zapallar},
\[
\begin{split}
\bb E_n [\|\mc M_t^n \|_{1,1}]
		&= \slinf \ell \,\bb E_n [ |\mc M_t^n(\ell)| ] \leq \slinf \ell \, \bb E_n [ \<\mc M^n(\ell)\>_t]^{1/2}
		\leq \slinf \ell \, \bb E_n \Big[ \int_0^t \!\! n \Gamma_n \pi_s^n(\ell) ds \Big]^{1/2}\\
		&\leq \slinf \Big( \| K \|_\infty  \int_0^t \bb E_n \Big[ \ell^2\big( \sil \pi_s^n(i) \pi_s^n(\ell-i) + 2\pi_s^n(\ell) \big) \Big] ds \Big)^{1/2} \\
		&\leq \Big(\slinf \frac{1}{\ell^2}\Big)^{1/2} \Big( \|K\|_\infty \int_0^t \bb E_n \Big[ \slinf \siinf (\ell+i)^4 \pi_s^n(\ell) \pi_s^n(i) + 2 \ell^4 \pi_s^n(\ell) \Big] ds \Big)^{1/2}\\
		&\leq C \|K\|_\infty^{1/2} \Big(\int_0^t 18 M_4(s)ds \Big)^{1/2},
\end{split}
\]
which is uniformly bounded in $[0,T]$ by Lemma \ref{l1}. Putting this estimate into \eqref{curico}, the lemma is proved.
\end{proof}

\subsection{Tightness}

As we did in the proof of the hydrodynamic limit, in order to prove tightness of the sequence $(\xi_t^n; t \in [0,T])_{n \in \bb N}$, we will use Aldous criterion. Condition (A1) is a simple consequence of Lemma \ref{l3}. In fact, observe that for every $M >0$, the set
\[
\big\{ u \in \ell^1(\bb N); \|u\|_{1,1} \leq M \big\}
\]
is compact in $\ell^1(\bb N)$. By Lemma \ref{l3},
\[
\bb P_n \big( \|\xi_t^n\|_{1,1} \geq M \big) \leq \frac{\bb E_n[\|\xi_t^n\|_{1,1}]}{M} \leq \frac{C(T)}{M},
\]
which proves (A1). The proof of (A2) is very similar to the proof we presented in Section \ref{s2.2}. By Dynkin formula,
\begin{equation}
\label{Dyn}
\xi_t^n = \int_0^t (\partial_s + L_n) \xi_s^n ds +  \mc M_t^n,
\end{equation}
so it is enough to prove tightness of the martingales $(\mc M_t^n; t \in [0,T])_{n \in \bb N}$ and of the integral processes $(\mc I_t^n; t \in [0,T])_{n \in \bb N}$ given by
\[
\mc I_t^n := \int_0^t (\partial_s + L_n) \xi_s^n ds
\]
for every $t \in [0,T]$ and every $n \in \bb N$. We start with the martingales. In order to proceed, 
let us estimate $\bb E_n[ \|\mc M_{\tau+\gamma}^n -\mc M_{\tau}^n\|_1]$:
\[
\begin{split}
\bb E_n[ \|\mc M_{\tau+\gamma}^n -\mc M_{\tau}^n
		&\|_1]
		= \bb E_n \Big[ \slinf \big| \mc M_{\tau+\gamma}^n(\ell) - \mc M_\tau^n(\ell)\big| \Big] 
		\leq \slinf \bb E_n \big[ \< \mc M^n(\ell)\>_{\tau}^{\tau+\gamma} \big]^{1/2} \\
		&\leq \Big( \slinf \frac{1}{\ell^2} \Big)^{1/2} \Big( \slinf \ell^2 \bb E_n \big[ \< \mc M^n(\ell)\>_{\tau}^{\tau+\gamma} \big]\Big)^{1/2}\\
		&\leq C \|K\|_\infty^{1/2} \bb E_n \Big[ \int_{\tau}^{\tau+\gamma} \slinf \ell^2 \Big( \sil \pi_s^n(i) \pi_s^n(\ell-i) + 2 \pi_s^n(\ell)  \Big) ds \Big]^{1/2}\\
		&\leq C \|K\|_\infty^{1/2} \bb E_n \Big[ \int_{\tau}^{\tau+\gamma} 6\slinf \ell^2 \pi_s^n(\ell) ds \Big] ^{1/2} \\
\end{split}
\]
Using Lemma \ref{l1}, we observe that for every $L \in \bb N$,
\[
\begin{split}
\bb E_n\Big[ \int_{\tau}^{\tau+\gamma} \slinf \ell^2 \pi_s^n(\ell) ds \Big] 
		&\leq C\gamma L^3 + \bb E_n \Big[ \int_0^T \slinf \frac{\ell^3}{L} \pi_s^n(\ell) ds \Big]\\
		&\leq C \gamma L^3 + \frac{1}{L} \int_0^T M_3(t) dt.
\end{split}
\]
Minimizing over $L$, we see that
\[
\bb E_n \Big[ \int_{\tau}^{\tau+\gamma} \slinf \ell^2 \pi_s^n(\ell) ds \Big] \leq C \gamma^{1/4} \Big( \int_0^T M_3(t) dt \Big)^{1/2},
\]
from where $(\mc M_t^n; t \in [0,T])_{n \in \bb N}$ satisfies (A2).
Now we are left to prove (A2) for $(\mc I_t^n; t \in [0,T])_{n \in \bb N}$.
Recall \eqref{portillo}. Using \eqref{colina}, we see that 
\[
\bb E_n \Big[ \Big\| \int_{\tau}^{\tau+\gamma}\frac{\mc R \pi_s^n}{\sqrt n} ds \Big\|_1 \Big]
		\leq \frac{3 \gamma \|K\|_\infty }{\sqrt n}.
\]
We conclude that this term is tight, and moreover that 
\begin{equation}
\label{rengo}
\lim_{n \to \infty} \int_{0}^{t}\frac{\mc R \pi_s^n}{\sqrt n} ds =0
\end{equation}
in probability with respect to the uniform topology. From \eqref{buin}, we see that
\[
\begin{split}
\big\| \sqrt n \big( \mcK \pi_s^n - \mcK u_s \big) \big\|_1 \leq 6 \|K\|_\infty \| \xi_s^n\|_1. 
\end{split}
\]
Therefore, using Lemma \ref{l2},
\[
\begin{split}
\bb E_n \Big[ \Big\| \int_{\tau}^{\tau+\gamma} \sqrt n  \big( \mcK \pi_s^n - \mcK u_s \big) ds \Big\|_1^2 \Big]
		&\leq \bb E_n \Big[  \Big( \int_{\tau}^{\tau+\gamma} 6 \|K\|_\infty \| \xi_s^n\|_1 ds \Big)^2\Big]\\
		&\leq 36 \|K\|_\infty^2 \gamma \int_0^T \bb E_n \big[ \|\xi_t^n\|_1^2\big] ds \leq C(T, \|K\|_\infty) \gamma,
\end{split}
\]
from where (A2) follows for $(\mc I_t^n; t \in [0,T])_{n \in \bb N}$ and in consequence for $(\xi_t^n; t \in [0,T])_{n \in \bb N}$..

\subsection{Convergence}

Now that we know that every subsequence $(\xi_t^n; t \in [0,T])_{n \in \bb N}$ has limit points, we are left to characterize these limits as solutions of \eqref{SDE}. The proof is very similar to the corresponding part of the proof of Theorem \ref{t2}. First we observe that the jumps of $(\xi_t^n; t \in [0,T])$ have size $\frac{1}{\sqrt n}$, from where we conclude that every limit point is continuous. Let $n'$ be a subsequence for which $(\xi_t^{n'}; t \in [0,T])_{n'}$ is convergent. Let $(\xi_t; t \in [0,T])$ the corresponding limit. Let $f \in \mc C^1([0,T]; \ell^1(\bb N))$. By Dinkyn formula, the process $(\mc M_t^n(f); t \in [0,T])$ given by
\begin{equation}
\label{curepto}
\mc M_t^n(f) := \<\xi_t^n,f_t\> - \int_0^t \big( \<\xi_s^n, \partial_s f_s\> + \<(\partial_s +L_n)\xi_s^n, f_s\> \big) ds
\end{equation}
for every $t \in [0,T]$, is a martingale of quadratic variation
\[
\< \mc M^n(f)\>_t = \int_0^t \Gamma_n \<\xi_s^n, f_s\> ds.
\]
Since $(\xi_t; t \in [0,T])$ has continuous trajectories, 
\[
\lim_{n' \to \infty} \< \xi_t^{n'}\!,f_t\> = \<\xi_t, f_t\>.
\]
Observe that
\[
\sqrt{n} \big( \mcK \pi_s^n - \mcK u_s \big) = \mc A(\xi_s^n, u_s).
\]
From \eqref{rengo}, we see that 
\[
\begin{split}
\lim_{n' \to \infty} \int_0^t \<(\partial_s + L_n) \xi_s^{n'}, f_s\>  ds 
		&= \lim_{n' \to \infty} \int_0^t \big\<\sqrt{n' } \big( \mcK \pi_s^{n'} - \mcK u_s \big), f_s \big\> ds\\
		&=\lim_{n' \to \infty} \int_0^t \<\mc A(\xi_s^{n'}\!, u_s), f_s \> ds.
\end{split}
\]
From \eqref{huasco}, we see that $\xi \mapsto \mc A(\xi,u_s)$ is Lipschitz continuous, uniformly in $t \in [0,T]$. Therefore, the \emph{a priori} estimate of Lemma \ref{l2} combined with the weak convergence of $(\xi_t^{n'}; t \in [0,T])_{n \in \bb N}$ shows that
\[
\lim_{n' \to \infty} \int_0^t \big\<\mc A(\xi_s^{n'}\!, u_s), f_s \big\> ds = \int_0^t \big\<\mc A(\xi_s,u_s), f_s \big\> ds.
\]

Once we have shown that the integral term in \eqref{curepto} converges, we are left to prove the convergence of the martingale term. Recall the definition of $Q$ given in \eqref{Q}. We have that
\[
\begin{split}
\Gamma_n \< \xi_t^n, f_t\> 
		&= \frac{1}{n^2} \sum_{x \in \Ln} \sum_{y \neq x} K(\eta_x, \eta_y)  \big( f_t(\eta_x +\eta_y) -f_t(\eta_x) -f_t(\eta_y) \big)^2\\
		&= \sum_{\ell, m \in \bb N} K(\ell,m) \Big( \pi_t^n(\ell) \pi_t^n(m) - \frac{1}{n} \mathds{1} (\ell=m) \pi_t^n(\ell) \Big) \big( f_t(\ell+m) - f(\ell) -f(m) \big)^2\\
		&= Q(\pi_t^n, f_t) + \frac{1}{n} \slinf \pi_t^n(\ell) \big( f_t(2\ell) - 2 f(\ell) \big)^2.
\end{split}
\]
Observe that
\[
\frac{1}{n} \slinf \pi_t^n(\ell) \big( f_t(2\ell) - 2 f(\ell) \big)^2 \leq \frac{ 9 \|f\|_\infty^2}{n}.
\]
By \eqref{Q2}, $Q(\pi_t^n, f_t)$ is continuous as a function of $\pi_t^n$. Therefore, by Theorem \ref{t1},
\[
\lim_{n' \to \infty} \<\mc M^{n'}\!(f)\>_t = \lim_{n' \to \infty} \int_0^t \Gamma_{n'} \<\xi_s^{n'}\!, f_s\> ds = \int_0^t Q(u_s,f_s) ds.
\]
By Theorem VIII.3.11 of \cite{JacShi}, we conclude that
\[
\lim_{n' \to \infty} \mc M_t^{n'}(f) = \mc M_t(f),
\]
where $(\mc M_t(f); t \in [0,T])$ is a continuous martingale of quadratic variation
\[
\<\mc M(f)\>_t = \int_0^t Q(u_s,f_s) ds.
\]

We conclude that every limit point of $(\xi_t^n; t \in [0,T])_{n \in \bb N}$ is a solution of \eqref{SDE} with initial condition $\xi_0=0$. By Proposition \ref{p2}, such solutions are unique, from where Theorem \ref{t2} follows.

\section*{Acknowledgements}

A.~A.~would like to thank the warm hospitality of IMPA. M.~J.~acknowledge CNPq for its support through the Grant Produtividade em Pesquisa 312146/2021-3 and the Grant Universal 408529/2025-3. M.~J.~acknowledges FAPERJ for its support through a CNE Grant.

\bibliographystyle{plain}


%
%
%
%
%
%
%
%
%
%
%
%
%
%
%
%
%
\end{document}